\documentclass{article}
\usepackage{amsmath,amssymb, amsthm, mathabx}

\newtheorem{thm}{Theorem}[section]
\newtheorem{lemma}[thm]{Lemma}

\newtheorem*{thm*}{Theorem}

\newtheorem{prop}[thm]{Proposition}

\newtheorem{cor}[thm]{Corollary}

\newtheorem{?}[thm]{Question}

\theoremstyle{definition}
 \newtheorem{definition}[thm]{Definition}

\newcommand{\N}{\omega}

\newcommand{\A}{\mathcal{A}}

\newcommand{\s}{\sigma}
\newcommand{\nin}{\notin}
\newcommand{\twow}{2^{<\omega}}

\newcommand{\conc}{{}^{\smallfrown}}

\newcommand{\dom}{\text{dom}}

\newcommand{\U}{\mathbb{U}}
\newcommand{\phie}{\phi_e}
\newcommand{\phies}{\phi_{e,s}}
\newcommand{\deltwo}{\Delta^0_2}

\newcommand{\LK}{\mathcal{LK}}
\newcommand{\KT}{\mathcal{KT}}
\newcommand{\emptystr}{\langle\rangle}
\newcommand{\M}{\mathcal{M}}
\newcommand{\Tis}{T_{i,s}}
\newcommand{\Sei}{S^e_i}

\newcommand{\lleq}{\kern.2em\raisebox{-.1em}{$\leqslant$}\kern-.4em{<}}

\renewcommand{\restriction}{\mathord{\upharpoonright}}
\newcommand\restr[2]{{% we make the whole thing an ordinary symbol
  \left.\kern-\nulldelimiterspace % automatically resize the bar with \right
  #1 % the function
  \restriction % pretend it's a little taller at normal size
  \right._{#2} % this is the delimiter
  }}

\title{On Reals with $\deltwo$-Bounded Complexity and Compressive Power}
\author{Ian Herbert\footnote{Department of Mathematics, National University of Singapore, Singapore; iherbert@nus.edu.sg}}
\date{}

\begin{document}

\maketitle

\begin{abstract} 
The (prefix-free) Kolmogorov complexity of a finite binary string is the length of the shortest description of the string. This gives rise to some `standard' lowness notions for reals: A is K-trivial if its initial segments have the lowest possible complexity and A is low for K if using A as an oracle does not decrease the complexity of strings by more than a constant factor. We weaken these notions by requiring the defining inequalities to hold up only up to all $\deltwo$ orders, and call the new notions \emph{$\deltwo$-bounded $K$-trivial} and \emph{$\deltwo$-bounded low for $K$}. Several of the `nice' properties of $K$-triviality are lost with this weakening. For instance, the new weaker definitions both give uncountable set of reals. In this paper we show that the weaker definitions are no longer equivalent, and that the $\deltwo$-bounded $K$-trivials are cofinal in the Turing degrees. We then compare them to other previously studied weakenings, namely \emph{infinitely-often $K$-triviality} and \emph{weak lowness for $K$} (in each, the defining inequality must hold up to a constant, but only for infinitely many inputs). We show that $\deltwo$-bounded $K$-trivial implies infinitely-often $K$-trivial, but no implication holds between $\deltwo$-bounded low for $K$ and weakly low for $K$. 

\end{abstract}
\section{Introduction}

%There are several approaches to the question of defining `randomness,' each arising from a particular intuition for what it means for something, say, a real number, to be random. These approaches are not mutually exclusive. There are often productive interactions between approaches, and some notions even have equivalent definitions according to more than one approach. Very roughly, the main paradigms are: randomness as typicality, where a random real is one that lies in all measure $1$ sets from a certain class (usually based somehow on definability); randomness as unpredictability, where a random real is one against which it is impossible to win arbitrarily much money by gambling on its bits; and finally randomness as complexity, where a random real is one that is difficult to describe. In this paper we focus on complexity, and examine reals that are in some sense `far from random' according to this paradigm. 

The prefix-free Kolmogorov complexity, $K(\s)$, of a binary string $\s$ is the length of the shortest self-delimiting program (in a given language) whose output is $\s$. We can extend this to a notion on reals by examining the complexities of all finite initial segments of their binary expansions. We say a real is \emph{Martin-L\"{o}f random} if the complexities of its initial segments are as high as possible, i.e., up to an additive constant $c$ we have for all $n$, $K(\restr{A}{n}) \geq n - c$, where $\restr{A}{n}$ is the initial segment of $A$ of length $n$. In this way we capture a notion of randomness that coincides with being difficult to describe. The Martin-L\"{o}f random reals are one of the cornerstones of the field of Algorithmic Randomness. On the other end of the spectrum, we have reals whose initial segment complexity is as low as possible. A string of length $n$ can always be used as a description of the number $n$, so the lowest complexity we can achieve is $K(n)$. We say real $A$ is $K$\emph{-trivial} if up to an additive constant $c$ we have for all $n$, $K( \restr{A}{n}) \leq K(n) +c$ (for an $n\in\N$, we use $K(n)$ to mean the complexity of a string of $n$ zeros). The $K$-trivials are another set of reals that are well-studied and of great interest in this field. 

Another way of comparing reals using Kolmogorov complexity is to examine their compressive power. By allowing programs to have oracle access to reals, we get a notion of relativized Kolmogorov complexity; the length of the shortest description of $\s$ that can use $A$ as an oracle. We can then compare the plain complexities of strings with their complexities relative to a given real to get some idea of the additional power that the real is providing to compression. Some reals, for example Martin-L\"{o}f randoms, have high compressive power, since used as an oracle they can give very short descriptions of their own initial segments, which are impossible to compress by an oracle-free program. However, we also get a concept of `lowness' for reals for this measure. We say a real is \emph{low for} $K$ if up to an additive constant $c$, for all finite binary strings $\s$, $K(\s) \leq K^{A}(\s)+c$, that is, $A$ provides no more than a constant amount of additional compression to any string. It is a remarkable fact due to Nies \cite{lowforkktrivial} that lowness for $K$ coincides exactly with $K$-triviality; having minimal complexity is the same as having minimal compressive power. However, intensionally the definitions are quite different and as we weaken the definitions slightly the notions come apart. We formalize some notation to be used throughout this paper. 

We use $\omega$ to denote the least countable ordinal, identified with the set of natural numbers. We use $\twow$ to denote the set of finite binary strings and  $2^{\omega}$ for the set of infinite binary sequences, identified with the binary expansions of reals. We use the symbol `$\conc$' to denote the operation of concatenation on $\twow$, omitting it where there will be no confusion, and the symbol `$\prec$' to denote the initial segment relation on $\twow \times \twow$ and $\twow \times 2^{\omega}$. We denote the restriction of an element $A \in 2^{\omega}$ to its finite initial segment of length $n$ by $\restr{A}{n}$. In contexts that involve both finite binary strings and natural numbers, we will use $\langle \s \rangle$ to denote the \emph{string} $\s$ as opposed to the natural number with decimal expansion $\s$ (i.e., $\langle 10 \rangle$ is the binary strings of $1$ followed by $0$, while $10$ is the natural number `ten'), unless this can be omitted without confusion. By a \emph{tree} we mean a subset of $\twow$ that is closed downwards under $\prec$. For such a tree $T$, we use $[T]$ to denote the set of infinite paths through $T$, i.e., $[T]=\{A\in 2^{\omega}: \forall n \ \restr{A}{n} \in T\}$. As stated above, we use $n$ to denote the string consisting of $n$ zeros. For computations or processes that may or may not converge, we use $\downarrow$ to denote convergence and $\uparrow$ to denote divergence. We use the symbol `$\leq^+$' to denote that an inequality holds up to an additive constant. We will use standard terminology and definitions for recursion theoretic concepts as given in for example \cite{soare} or \cite{}.

By a \emph{machine} we mean a partial recursive function $\M: \twow \rightarrow \twow$. A machine $\M$ is \emph{prefix-free} if for any $\sigma \prec \tau$ in $\twow$, if $\M(\sigma)\downarrow$ then $\M(\tau)\uparrow$. For a prefix-free machine $\M$, the \emph{prefix-free Kolmogorov complexity relative to} $\M$ of a string $\s$ is $\min\{|\tau|:\M(\tau)=\s\}$ and is denoted $K_{\M}(\s)$. Solomonoff, Kolmogorov, and Chaitin each independently showed the existence of universal prefix-free machines, that is machines $\M_U$ such that for any other prefix-free machine $\M$, for all $\sigma \in \twow$, $K_{\M_U}(\sigma) \leq^{+} K_{\M}(\sigma)$. We fix some such universal prefix-free machine and denote it $\U$ and the associated Kolmogorov complexity simply $K$. For many of our proofs we will need to construct our own machines, and we will need the following result. A \emph{Kraft-Chaitin set} is a recursively enumerable subset $W$ of $\twow \times \omega$ such that $\sum\limits_{(\s,n)\in W}2^{-n} <1$. The Kraft-Chaitin Theorem, which appeared independently in work of Levin, states that for any such set $W$ there exists a prefix-free machine $\M$ such that for any pair $(\s, n)\in W$ there is a $\tau \in \twow$ such that $|\tau|=n$ and $\M(\tau)=\s$. 

In this paper the main objects of study are the weakenings of the standard lowness notions discussed above derived from replacing the constants with slow-growing functions. 

\begin{definition}$\ $\\
\begin{itemize}
\item For a function $f: \twow \rightarrow \N$, a real $A$ is \emph{low for K up to f} if for all $\s \in \twow$
$$K(\s) \leq^+ K^{A}(\s) +f(\s).$$

\item  For a function $g:\N \rightarrow \N$, a real $A$ is $K$\emph{-trivial up to g} if for all $n \in \N$ 
$$K(\restr{A}{n}) \leq ^+ K(n) +g(n)$$
\end{itemize}
\end{definition}
We write $\LK(f)$ for the set of reals that are low for $K$ up to $f$, and $\KT(g)$ for the set of reals that are $K$-trivial up to $g$. In this notation $\KT(0)$ is the set of standard $K$-trivials and $\LK(0)$ is the set of standard lows for $K$ (so $\KT(0)=\LK(0)$).

The question now arises as to which functions it will be fruitful to consider for these $f$ and $g$. Obviously some functions grow quickly enough that $\LK(f)$ or $\KT(g)$ is all of $\twow$. On the other hand, many functions (any with a finite $\limsup$) will just give us $\KT(0)$ or $\LK(0)$ again. As these functions represent the rates of growth of some quantities, it is natural to consider \emph{orders}, that is, functions that are unbounded and nondecreasing (some sources additionally require that orders be recursive, but we make no such restriction). In principle we can consider orders of arbitrarily high arithmetic complexity. However, Csima and Montalb{\'a}n showed that there is a $\Delta^{0}_{4}$ order $f$ such that $\KT(f)=\KT(0)$, that is, $A$ is $K$-trivial if and only if $K(\restr{A}{n})\leq^+ K(n)+f(n)$ \cite{csimamontalbangap}. Later Baartse and Barmpalias improved this by constructing a $\Delta^{0}_{3}$ order with this property \cite{baartsebarmpalias}, and showed that no such so-called `gap function' could be $\deltwo$. Thus, the $\deltwo$ order case is where these weakenings can be interesting and can be handled in a general way. It will often be more convenient in the proofs to work with a slightly more general notion than being a $\deltwo$ order, which we define below.

%A \emph{recursive approximation} is a recursive function $f\colon \N \times \N \rightarrow \N$. Such an $f$ \emph{approximates} a total function $g$ (denoted $f\rightarrow g$) if for all $n$, $\lim_{s\rightarrow \infty} f(n,s)=g(n)$. We write $f_s$ for the recursive function on $\N$ given by fixing the second argument of $f$ to be $s$, and we understand this to be the stage $s$ approximation to $g$. 

For a total function $f:\N \rightarrow \N$, a \emph{recursive approximation} is a uniformly recursive series of functions $(f_s)$ such that for all $x$, $\lim_{s\rightarrow \infty} f_s(x)=f(x)$. By the Schoenfield Limit Lemma and Post's Theorem (both in, for example \cite{soare}) a function has a recursive approximation if and only if it is $\deltwo$. We use some effective listing of all partial recursive approximations and write $\phi_{e,s}$ for the $s$th stage of the $e$th approximation.  

\begin{definition}
A $\deltwo$ function $f \colon \N \rightarrow \N$ is \emph{finite-to-one approximable} if it is total and has a recursive approximation $f_s \rightarrow f$ such that for any $n\in \N$, for all but finitely many $m\in N$, for all $s$, $f_s(m) >n$. Such an approximation is called a \emph{finite-to-one approximation}.
\end{definition}

We note that this is a more restrictive notion than having an approximation that is finite-to-one at each stage (any $\deltwo$ function will have such an approximation). We require rather that for a given output there are only finitely many inputs that are \emph{ever} in its preimage (so the function is finite-to-one over the whole approximation). With a simple diagonalization one can even construct a  finite-to-one $\deltwo$ function that fails to have a finite-to-one approximation in the above sense, which necessitates the complication of terminology.
% We will use the $\FOA$ to denote the index set $\{ e: \phies\text{ is a finite-to-one approximation}\}$.  

Finite-to-one approximability may seem like an odd condition to impose, but in some sense it is a generalization of being an order. Any $\deltwo$ order is finite-to-one approximable, by taking any recursive approximation and selecting only the stages where it looks like an order on initial segments of increasing length (and replacing the tail with the identity, if necessary). Moreover, any finite-to-one approximable function pointwise dominates some $\deltwo$ order. Since each $n$ will only ever appear in the output for finitely many inputs, each time it does so we can drop the value on all smaller inputs to $n$ to maintain monotonicity. Eventually we reach a point where $n$ never appears again, so our new function will have $\liminf$ greater than $n$.

A central concept of this paper will be those reals that are $K$-trivial or low for $K$ up to \emph{every} $\deltwo$ order.  We use
$$\KT(\deltwo)=\bigcap\limits_{f \text{ a } \deltwo \text{ order}}\KT(f),$$
to denote the set of reals that are $K$-trivial up to every $\deltwo$ order and
$$\LK(\deltwo)=\bigcap\limits_{f \text{ a } \deltwo \text{ order}}\LK(f).$$
to denote the set of reals that are low for $K$ up to every $\deltwo$ order. By the discussion above it should be clear that these coincide with the reals that are $K$-trivial or low for $K$ up to every finite-to-one approximable function. 

Since no Martin-L\"{o}f random real can be $K$-trivial or low for $K$ up to even $\log(n)$, it is clear that both of these sets have measure $0$. The reals in $\KT(\deltwo)$ are within every $\deltwo$ order of being $K$-trivial. Another way to think of these reals is that, while they may not be $K$-trivial, there is no $\deltwo$ witness to their non-$K$-triviality. Any function $f\colon \N \rightarrow \N$ such that $f(c)$ gives an $n$ with $K(\restr{A}{n}) > K(n)+c$ can not be $\deltwo$, and the analogous statement holds for $\LK(\deltwo)$. Since the function $K(\s)$ is itself $\deltwo$, among the $\deltwo$ reals the only reals in $\KT(\deltwo)$ or $\LK(\deltwo)$ must be the $K$-trivials. 

\begin{prop}
If $A$ is $\deltwo$ and not $K$-trivial, then $A\nin \KT(\deltwo)$ and $A\nin \LK(\deltwo)$. 
\end{prop}

\begin{proof}
The proof is almost trivial, except for one slight complication. For a $\deltwo$ real $A$, the function $K(\restr{A}{n})-K(n)$ is also $\deltwo$. However, it need not be finite-to-one approximable. For example, for every r.e. real this difference has a finite $\liminf$, by results of Barmpalias and Vlek \cite{barmpaliasvlek}.

Given a $\deltwo$ non-$K$-trivial real $A$, we define $g_s(n)$ to be $2^{l-1}$ if $n\leq s$, where $l$ is the greatest such that for some $m\leq n$, $K_s(\restr{A_s}{m})-K(m) > 2^{l}$, using $2^{-1}=0$, and $g_s(n)=n$ for $n >s$. Eventually $\restr{A}{n}$, $K(\restr{A}{m})$, and $K(m)$ for all $m\leq n$ all converge, so $g=\lim g_s$ is total. 

Since $A$ is not $K$-trivial, for each $l$ there will be some $n$ such that $K(\restr{A}{n})-K(n)>2^{l}$, and so after the first stage $s$ for which these values have all converged no number greater than $n$ will ever get a $g$-value less than $2^{l-1}$. Thus, the approximation $g_s \rightarrow g$ is a finite-to-one approximation. It is clear that $A \nin \KT(g)$, since for infinitely many $l$ there is an $n$ with $K_s(\restr{A}{n})-K(n) > 2^{l}$ but $g(n)\leq 2^{l-1}$. Therefore, $A$ cannot be in $\KT(\deltwo)$

Similarly, if $A$ is $\deltwo$ then so is the function $K(\s)-K^{A}(\s)$. In case this difference is not finite-to-one approximable, we can define the function $f_s(\s)=2^{n-1}$ for the largest $n$ such that $\exists \tau$ with $|\tau|\leq |\s|$ and $K_s(\tau)-K^{A}_s(\tau)\geq 2^n$ if $\s$ is one of the first $s$ strings in length-lexicographic order and $f_s(\s)=|\s|$ if $\s$ is not one of these strings. This $f_s\rightarrow f$ will be finite-to-one approximable, and $A$ will fail to obey $\forall \s$, $K(\s)\leq^+K^{A}(\s)+f(\s)$.

%Reals for which this function has a finite $\liminf$ are called \emph{weakly low for} $K$, and Miller has shown in \cite{millerweaklowk} that these coincide exactly with the low for $\Omega$ reals defined by Nies, Stephan, and Terwijn \cite{lowforomega}. 

\end{proof}

In this paper we examine some properties of $\KT(\deltwo)$ and $\LK(\deltwo)$. In particular, we are interested in how these sets compare to each other, to the standard notions of lowness for $K$ and $K$-triviality, and to some other weakenings of these notions. We first review some results about the classic case and what is known about certain weakenings. 

In the classic case, as noted above it is a result of Nies that $\LK(0)=\KT(0)$ \cite{lowforkktrivial}. $\LK(0)$ (and thus $\KT(0)$) is clearly closed downwards under $\leq_T$, since a real can simulate the compression done by any real it can compute. Chaitin \cite{chaitin} showed that $\KT(0)$ was countable and all its members are $\deltwo$. Downey, Hirshfeldt, Nies, and Stephan have shown that $\KT(0)$ is closed under effective join \cite{downhirschniessteph}. 

In contrast Baartse and Barmpalias \cite{baartsebarmpalias} constructed for any $\deltwo$ order $g$ a perfect set of reals in $\KT(g)$, so this set is uncountable and has non-$\deltwo$ elements. Hirschfeldt and Weber, though they did not use this terminology, first showed that for any finite-to-one approximable $f$, $\LK(f)$ contains an r.e. set that is not in $\LK(0)$ \cite{hirschfeldtweber}. In an earlier paper \cite{herbert}, the author has shown that there is a perfect set of reals in $\LK(f)$ for any finite-to-one approximable $f$, and in fact in $\LK(\deltwo)$.  Additionally, the perfect set constructed in that paper has the property that for any real $A$ there are two elements $B_1$ and $B_2$  of the set such that $B_1 \oplus B_2 \geq_T A$, so in general $\LK(f)$ and $\LK(\deltwo)$ are not closed under effective join (i.e., except for the trivial case when $\LK(f)=2^{\omega}$). 

%%%%%%%%%%%%%%%%%%%  REDO WITH RIGHT SECTION NUMBER%%%%%%%%%%%%%%%%%%%%%%%5
In Section 2 we discuss the downwards closure of $\KT(\deltwo)$ under $\leq_T$ and implication between $\KT(\deltwo)$ and $\LK(\deltwo)$. In Section 3 we give some positive closure results for $\KT(\deltwo)$, and in Section 4 we compare these notions with other lowness notions related to Kolmogorov complexity and closures under weaker reducibilities. We give further directions for study in Section 5. 

%%%%%%%%%%%%%%%%%%%%%%%%%%%%%%%%%%%%%%%%%%%%%%%%%%%%%%%%%%%%%%%%

%\section{A failure of Nies's Theorem in the weaker case}

%%%%%%%%%%%%%%%%%%%%%%%%%%%%%%%%%%%%%%%%%%%%%%%%%%%%%%%%%%%%%%%%%%%%
%%%%%%%%%%%%%%%%%%%%%%%%%%%%%%%%%%%%%%%%%%%%%%%%%%%%%%%%%%%%%%%%%%%%
%%%%%%%%%%%%%%%%%%%%%%%%%%%%%%%%%%%%%%%%%%%%%%%%%%%%%%%%%%%%%%%%%%%%
%%%%%%%%%%%%%%%%%%%%%%%%%%%%%%%%%%%%%%%%%%%%%%%%%%%%%%%%%%%%%%%%%%%%
%%%%%%%%%%%%%%%%%%%%%%%%%%%%%%%%%%%%%%%%%%%%%%%%%%%%%%%%%%%%%%%%%%%%
%%%%%%%%%%%%%%%%%%%%%%%%%%%%%%%%%%%%%%%%%%%%%%%%%%%%%%%%%%%%%%%%%%%%
%%%%%%%%%%%%%%%%%%%%%%%%%%%%%%%%%%%%%%%%%%%%%%%%%%%%%%%%%%%%%%%%%%%%
%%%%%%%%%%%%%%%%%%%%%%%%%%%%%%%%%%%%%%%%%%%%%%%%%%%%%%%%%%%%%%%%%%%%
%%%%%%%%%%%%%%%%%%%%%%%%%%%%%%%%%%%%%%%%%%%%%%%%%%%%%%%%%%%%%%%%%%%%
\section{Downwards closure}

It is easy to see that lowness for $K$ as a property of reals is closed downwards under $\leq_T$, since a real can simulate the compression algorithms of any real it can compute. From this and the Nies's Theorem it follows that $K$-triviality is also closed downwards. The same argument shows that $\LK(f)$ is closed downwards for any $f$, but we do not have an analog of Nies's Theorem in the weaker case, so it does not necessarily follow that $\KT(g)$ is closed downwards in general. In fact, this is not the case. The downwards closure under $\leq_T$ fails in a very strong sense, with $\KT(\deltwo)$ (and hence each $\KT(g)$ for $\deltwo$ order $g$) actually being cofinal in the Turing degrees.

\begin{thm}\label{anybktrivall}
For any real $B$ there is an $A\geq_{T} B$ such that $A \in \KT(\deltwo)$.
\end{thm}

\begin{proof}
We wish to build an $A$ that is Turing-above $B$ and that is $K$-trivial up to every finite-to-one approximation. We do not know \emph{a priori} which $\phies$ are finite-to-one approximations, and, since we need Kraft-Chaitin sets to be recursively enumerable, we will have to build a tree $T$ and use the branching nodes to mark guesses as to the behaviors of the $\phies$ 's. The tree we build will be independent of $B$ and will contain a witness $A$ for every real $B$. 

Placing the branching nodes in our tree will be a delicate operation. We will use a system of markers $\gamma(\alpha)$ to keep track of the values where corresponding $\phies$'s are large enough to place another branching node, which we will use to guess the behavior of the next $\phi_{e+1,s}$. Additionally, we will use these markers to mark `coding locations' where the bits of a given real can be stored, which will make the behavior of the tree around these nodes slightly more complicated. We will have to introduce another kind of branching node which will keep track of which value is in the $i$th bit of $B$. Once we put a marker $\gamma(\alpha)_s$ at a node to make it a guessing node, both successors of that node, $\gamma(\alpha)_s\conc 1$ and $\gamma(\alpha)_s\conc 0$, will also both be in $T_s$, and a path taking one or the other of these nodes will correspond to guessing whether $\phi_{|\alpha|}$ is or is not a finite-to-one approximation. After this branching we will immediately introduce another branching, so that $\gamma(\alpha)_s\conc 11$, $\gamma(\alpha)_s\conc 10$, $\gamma(\alpha)_s\conc 01$, and $\gamma(\alpha)_s\conc 00$ will all be in $T_s$. The value of a path through $\gamma(\alpha)$ at $|\gamma(\alpha)|+2$ will correspond to the $|\alpha|$th bit of $B$. We distinguish between the two kinds of branching nodes as either \emph{guessing nodes} or \emph{coding nodes}. To \emph{kill} a node is to make a commitment never to add nodes to $T$ above it, and a node that has not yet been killed is \emph{living}.

%There is one final way for both extensions of a node to appear in $T$, which is when we kill the tree above $\gamma(\alpha)$ and start building above $\gamma(\alpha)^{-}\conc 1$. This makes $\gamma(\alpha)^-$ a branching node in $T$, but at no stage of the construction are both nodes above $\gamma(\alpha)^-$ alive (and thus in $T$ the subtree above one of these nodes is finite), so we will not refer to these as branching nodes. 

We make some definitions for ease of bookkeeping. A path that follows $\alpha$ through the first $|\alpha|$-many branching nodes is only making guesses about the first $|\alpha|/2$-many $\phies$'s, since only the even-numbered branching nodes correspond to guesses. We say that $e$ \emph{is a guessing member} of $\alpha$ if and only if $\alpha(e)=1$ and $e$ is even. We denote this $e \in_g \alpha$.  For each $\alpha\in \twow$, we use $\psi_{\alpha,s}(n)=\min\{\phi_{e,s}(n): e \in_g \alpha\}$ to denote the function that makes all of $\alpha$'s guesses. Note that $\psi_{\alpha}$ considers up to $|\alpha|/2$ many $\phies$'s. We will need to ensure we branch at a rate that forces these values to be large enough that we can afford to pay for multiple initial segments of the same length into our Kraft-Chaitin sets. We will want to ensure that each $\psi_{\alpha,s}$ takes values at least as large as $2|\alpha|$ before we add another branching. 

We build a separate Kraft-Chaitin set $M_e$ to witness $A$'s $K$-triviality up to $\phi_e$. $M_e$ will only take requests to describe initial segments that are on the tree above nodes that guess that $\phies$ is a total finite-to-one approximation, so for a correct path $A$ through $T$ these $M_e$ together will witness that $A \in \KT(\deltwo)$.  

Ensuring that there is such an $A$ that is Turing-above $B$ will be handled after the construction. Essentially, we will just show that the path through the `true subtree' that contains the bits of $B$ in its coding locations can find these locations recursively  

The requirements we are trying to meet are

\begin{flalign*}
B_{\alpha}: \ &\text{The path through } T \text{ that follows } \alpha \text{ through the branching nodes branches}& \\ 
&\text{twice more at a level } n \text{ where } \psi_{\alpha}(n) \geq 2|\alpha|&\\
\end{flalign*}
for all $\alpha\in\twow$ with $|\alpha|$ even, 

\begin{flalign*}
&R^e_i: \ \ \text{ For all } n \text{ with } 2i\leq\phi_e(n)< 2i+2, K(\restr{A}{n}) \leq^+ K(n) +2i&
\end{flalign*}
for all $i, e \in \N$ with $e\leq i$, and

We order these requirements $B_{\emptystr}, R^0_0, B_{\langle 00 \rangle}, B_{\langle 01 \rangle},  B_{\langle 10 \rangle}, B_{\langle 11 \rangle}, R^0_1, R^1_1, B_{\langle 0000 \rangle}, \hdots$. The construction will be an injury construction, and we give the strategies for meeting each of the requirements. 

A $B_{\alpha}$ requirement will \emph{require attention} at a stage $s$ if there is not a living branching node $\tau$ above the path that follows $\alpha$ through $T_s$ with $\psi_{\alpha,s}(|\tau|) \geq 2|\alpha|$. The strategy for satisfying $B_{\alpha}$ is
\begin{enumerate}
\item Search for an $n$ such that $\psi_{\alpha,s}(n) \geq 2|\alpha|$
\item Extend the longest path that follows $\alpha$ with a string of $0$s to a length $n'+1$ where $n'>n$ has not been used yet in the construction. Put the marker $\gamma(\alpha)_{s+1}$ at the node on this branch of length $n'+1$. Put both extensions of length $n'+2$ and all four extensions of this node of length $n'+3$ into $T_{s+1}$. 
\end{enumerate}

An $R^e_i$ requirement will \emph{require attention} at a stage $s$ if there is an $n$ with $2i\leq \phies(n)< 2i+2$ and there is a living path $\tau$ through $T_s$ of length at least $n$ and there is an $\alpha$ such that $e$ is a guessing member of $\alpha$ and $\gamma(\alpha)_s \preceq \tau$ such that there is not a request in $M_{e,s}$ for a description of $\restr{\tau}{n}$ of length less than $K_s(n)+2i$. The strategy for satisfying $R^e_i$ is
\begin{enumerate}
\item For all such $n$ and $\tau$, for the longest $\alpha$ such that $\gamma(\alpha)_s\preceq \tau$ and $e$ is a guessing member of $\alpha$, put the request $(\restr{\tau}{n}, K_s(n)+2|\alpha| )$ into $M_{e,s+1}$.

\end{enumerate}

We now give the construction, which will call these subroutines as necessary.

\textbf{Stage 0}: $T_0=\emptyset$, $M_0=\emptyset$, $\gamma(\emptystr)_0=\emptystr$ and $\gamma(\alpha)_0$ undefined for all other $\alpha$.\\

\textbf{Stage s+1}: \\
\begin{enumerate}
\item Compute $\phi_{e,s+1}(n)$ for $e,n \leq s+1$.

\item If there are an $n$ and an $\alpha$ such that $|\gamma(\alpha)_s| < n$ and $\psi_{\alpha,s+1}(n)<2|\alpha|$, then $\gamma(\alpha)_s$ is no longer marking a point after which $\psi_{\alpha}$ is greater than $2|\alpha|$, so for the length-lexicographically first $\alpha$

\begin{enumerate}
\item Kill all branches of the tree above $\gamma(\alpha)_s$. 
\item Let $\gamma(\alpha)_s^{-}$ be the initial segment of $\gamma(\alpha)_s$ of length $|\gamma(\alpha)_s|-1$ and put $\gamma(\alpha)_s^-\conc 1$ into $T_{s+1}$ as a living node. Note that again by construction $\gamma(\alpha)_s$ always ends in a $0$ and is always a node of length at least 2 longer than any number seen earlier in the construction, hence $\gamma(\alpha)_s^-\conc 1$ will not have been used before this point. 

\item Repeat 2a) and 2b) for all other such $\alpha$ with $\gamma(\alpha)_s$ still living, in length-lexicographic order. 
\end{enumerate}

\item For the highest priority requirement that requires attention of the first $s+1$ many requirements, run $s+1$-many steps of its strategy.
\item Repeat 3.) for any of the first $s+1$-many requirements that still require attention, in order of decreasing priority.

\end{enumerate}

This completes the construction. We let $T=\bigcup\limits_s T_s$, $M_e=\bigcup\limits_s M_{e,s}$, and $\gamma(\alpha)=\lim\limits_s \gamma(\alpha)_s$. Unfortunately, in this construction, as in life, not all requirements can be satisfied. We call a string $\alpha$ \emph{correct} if for all $n \leq |\alpha|$, if $n=2m$ is even then $\alpha(n)=1$ if and only if $\phi_{m,s}$ is a finite-to-one approximation. We need to show that the requirements that are relevant to building correct paths through $T$ are all satisfied, and that the mass we put into $M_e$ to satisfy the $R^e_i$ requirements is bounded.

\begin{lemma} 
The requirements $B_{\alpha}$ for correct $\alpha$ and  $R^e_i$ for $e$ such that $\phies$ is a finite-to-one approximation are all eventually satisfied. 
\end{lemma}

\begin{proof}
First we argue that each of these requirements can be subject to at most finitely many injuries. An injury to any $B_{\alpha}$ requirement only occurs when for some $e \in_g \alpha$ the value of $\phies(n)$ drops below $2|\alpha|$ for some $n>|\gamma(\alpha)_s|$. When this happens, we respond by moving $\gamma(\alpha)_{s+1}$ to a higher level. Now, since there are only finitely many $e\in_g \alpha$, if there were infinitely many such injuries then at least one $e\in_g \alpha$ would be responsible for infinitely many. There would then have to be infinitely many $n$ such that for some $t$, $\phi_{e,t}(n) <2|\alpha|$, so $\phies$ would not a finite-to-one approximation. Then $e \in_g \alpha$ is a contradiction to $\alpha$'s correctness. 

For $R^e_i$ requirements, since $\phies$ is a total finite-to-one approximation we will eventually reach some stage $s$ where $\phies$ has converged on all $n$ such that $\phi_e(n)<2i+2$. At this stage all the $n$ that $R^e_i$ will ever be concerned about have been found. We let the largest of these $n$ be $n'$. Then injuries to $R^e_i$ can only occur either when $K_s$ changes for one of these $n$, but this happens only finitely often, or when there is a change in $T$ below $n'$ on some branch that is guessing that $\phies$ is a total finite-to-one approximation. Each of these changes moves some marker to a point larger than $n'$, and, since there were only finitely many markers at positions lower than $n'$ at stage $s$, this too can only happen finitely often. 

Now we need to show that once these requirements are no longer injured they will be able to act to satisfy themselves. For $B_{\alpha}$ requirements, the strategy waits until it finds an $n$ such that $\psi_{\alpha,s}(n)\geq 2|\alpha|$, and then extends a path in the tree to this height and branches twice. Since every $e\in_g \alpha$ is in fact a total finite-to-one approximation, there will exist an $n$ for which $\psi_{\alpha}(m) \geq 2|\alpha|$ for all $m>n$, and so eventually $B_{\alpha}$ will find such an $m$ and act and be satisfied. 
For $R^e_i$ requirements, the strategy puts requests into $M_e$. After it is no longer injured, it needs to act at most once for each $n$ with $\phi_e(n)<2i+2$ and there are only finitely many of these so it is eventually satisfied. 

\end{proof}

The $R^e_i$ requirements for $\phies$ that are finite-to-one approximations will act until satisfied, but this happens just by placing the relevant requests into $M_e$. We now need to find an upper bound on the mass put into $M_e$ in this way, to ensure we get a machine that serves our purpose.     
 
\begin{lemma}
For all $e$, $\mu(\emph{dom}(M_e))\leq 8$.
\end{lemma}

%  We may use the total number of paths in $T$ of length $n$ as a rough bound on the number of times we will need to pay into $M_e$ of behalf of $n$, even though we will only pay on those that guess that $\phies$ is a total finite-to-one approximation. 

\begin{proof}
Let us start by fixing an $e$. We will bound the amount paid into $M_e$ for an arbitrary $n$ using some description $\s$ of $n$ from $\U$. For any path $\tau$ of length $n$ in $T$, while $\tau$ is alive there is some maximal $\alpha$ such that $\gamma(\alpha)_s \preceq \tau$. Now, a given $\alpha$ can be maximal for at most 4 such $\tau$'s since $T_s$ will branch twice immediately after $\gamma(\alpha)_s$ and then no more until $\gamma(\alpha')$ for some $\alpha' \succ \alpha$, at which point $\alpha'$ would be maximal. A change in $T_s$ below $n$ moves at least one of these markers to some new height above $n$, at which point a smaller $\alpha$ becomes maximal. Recall that $\gamma(\alpha)$ are only placed for $\alpha$ of even length. 

For each $\alpha$ of even length, we know that $\gamma(\alpha)_s$ is placed at a node such that $\psi_{\alpha,s}(|\gamma(\alpha)_s|) \geq 2|\alpha|$, and so we know the rate at which we will pay into $M_e$ for $\tau$ for which $\alpha$ is maximal is $2^{-2|\alpha|}$. Here we use the fact that we only pay for paths with $e \in_g \alpha$. Now all that remains is to add up the total mass that could be paid into $M_e$ when any $\alpha$ is maximal such that $\gamma(\alpha) <n$. This gives us the sum

$$\sum\limits_{\substack{|\alpha| \text{ even }\\0\leq|\alpha|<\infty}} 4\cdot 2^{-|\s|}2^{-2|\alpha|}.$$

Since there are $2^{2i}$-many $\alpha$ of length $2i$ this can be rewritten as

$$4\cdot \sum\limits_{i=0}^{\infty} 2^{2i}\cdot 2^{-|\s|-4i}.$$

This sum reduces to $4\cdot 2^{-|\s|}\cdot \sum\limits_{i=0}^{\infty} 2^{-2i}$, which is bounded by $2$, so we can bound the mass paid into $M_e$ on behalf of $n$ using $\s$ by $8\cdot 2^{-|\s|}$. We now sum over all $\s$ in the domain of the universal machine to find a bound of the mass paid into $M_e$ for any $n$ using any $\s$. We get that this mass is bounded by $\sum\limits_{\s \in \dom(\U)} 8\cdot 2^{-|\s|} \leq 8$, and we are done. 

\end{proof}

This lemma gives us that a path through $T$ that guesses correctly that $\phies$ is a finite-to-one approximation will be $K$-trivial up to $\phi_e$. Thus, paths through $T$ that guess correctly about the behavior of every $\phies$ (i.e. that are \emph{correct}) will be $K$-trivial up to every finite-to-one approximation. Now all that remains is to show that given a real $B$ we can find a path through this true subtree of $T$ that computes $B$. 

\begin{lemma}
For any $B$, there is an $A\in[T]$ such that $A\geq_{T} B$ and $A$ is correct. 
\end{lemma} 

\begin{proof}
Suppose we are given $B$. We define a sequence of strings $\alpha_i$. Let $\alpha_i$ be such that $|\alpha_i|=2i$ and for all $n< 2i$, if $n=2m+1$ then $\alpha_{i}(2m+1)=B(m)$ and if $n=2m$ then $\alpha_i(n)=1$ if and only if $\phi_{m,s}$ is a finite-to-one approximation. Since $\alpha_i$ is correct and $|\alpha_i|$ is even, $B_{\alpha_i}$ is a requirement in our construction that is eventually satisfied. This means that we will eventually place the marker $\gamma(\alpha_i)$ on some node that follows $\alpha_i$ through the first $2i$ many branching nodes, where it will remain for the rest of the construction. Now $\alpha_i\prec \alpha_{i+1}$ for all $i$, so $\gamma(\alpha_i)$ must necessarily be an initial segment of $\gamma(\alpha_{i+1})$ for all $i$. Then we can let $A=\bigcup\limits_{i} \gamma(\alpha_i)$ and this is well-defined. Now, $A$ follows $\alpha_i$ through $T$ for every $i$, and each $\gamma(\alpha_i)$ is correct about its guesses, so $A$ must also be correct about its guesses. From this and the previous lemma, we know that $A \in \KT(\deltwo)$. All that remains is to show that $A \geq_{T} B$. 

First, we know that the bits of $B$ are encoded somewhere in $A$, since they are the odd bits of the $\alpha_i$'s. $A$ follows $\alpha_i$ through the branching nodes of $T$, so $A(|\gamma(\alpha_i)|+2)=B(i)$. $A$ can simulate the construction of $T$ and the approximation to $\gamma(\alpha)$ for each $\alpha$. When we move a $\gamma(\alpha)$ we first kill the tree above $\gamma(\alpha)_s$ and then start building above $\gamma(\alpha)^{-}_s\conc 1$. Thus, $A$ can tell when it finds a $\gamma(\alpha_i)_s$ whether this will be the location of $\gamma(\alpha_i)$ at the end of the construction. When it reaches a stage such that $\gamma(\alpha_i)_s \prec A$ it knows this is the final location of $\gamma(\alpha_i)$, and so it can retrieve the $i$th bit of $B$.

\end{proof}

This was the final step in the proof of Theorem~\ref{anybktrivall}.
\end{proof}

We end this section with a few remarks on the proof. First, we note that the paths through $T$ that we construct to compute $B$ may have much higher Turing degree than is necessary. By the same process that $A$ uses to compute the bits of $B$, $A$ can deduce which $\phies$ are total finite-to-one approximations and the index set $\{ e : \phies \text{ is a total finite-to-one approximation } \}$ is $\Pi^0_3$-complete. Of course, for certain $B$ there may be much less complicated $A$ (for example, if $B$ is recursive then $A=B$ is in $\KT(\deltwo)$). 

%Theorem~\ref{anybktrivall} shows that not only is $\KT(\deltwo)$ not closed downwards under $\leq_{T}$ but there is not even an upper bound on the Turing degrees that contain elements of $\KT(\deltwo)$. 

Turning to analogues of Nies's Theorem (that $K$-triviality and lowness for $K$ coincide) with these weaker notions, we see that one direction still holds, and in fact follows easily.

\begin{prop}
If $A\in \LK(\deltwo)$, then $A\in \KT(\deltwo)$.
\end{prop}

\begin{proof}
For any $f\colon \N \rightarrow \N$, let $\hat{f}\colon \twow \rightarrow \N$ be given by $\hat{f}(\s)=f(|\s|)$. Clearly if $f$ is a $\deltwo$ order then so if $\hat{f}$. If $A\in \LK(\deltwo)$, then for any $\deltwo$ order $f$ we have $\A\in \LK(\hat{f})$, i.e. for some $c$, for all $\s\in\twow$, we have $K(\s)\leq K^{A}(\s)+\hat{f}(\s)+c$. In particular we get for all $n$, $K(\restr{A}{n})\leq K^{A}(\restr{A}{n})+f(n)+c$. Now, relative to $A$ there is a very short description of $\restr{A}{n}$: read off the first $n$ bits of the oracle. All this machine requires to produce this initial segment is a description of the number $n$, so $K^{A}(\restr{A}{n})\leq^+ K^{A}(n)$, and this can be no larger than $K(n)$. Thus, for all $n$, $K(\restr{A}{n})\leq^+K(n)+f(n)$, so $A\in \KT(f)$. Since this holds for any $\deltwo$ order $f$, $A$ must be an element of $KT(\deltwo)$. 
\end{proof}

%We also note that Theorem~\ref{anybktrivall} implies a stronger version of Theorem~\ref{lowforknotktrivall}, which said that there was a recursive $f$ such that $\KT(\deltwo)$ is not contained in $\LK(f)$. The only important property of this $f$ was its growth rate (in fact, that there were sufficiently many preimages for each $i$) and so a similar proof technique would work for a variety of functions. From Theorem~\ref{anybktrivall} and the fact that $\LK(f)$ \emph{is} closed downwards under $\leq_{T}$ we get the following dichotomy without having to worry about adjusting the earlier proof. 

In the other direction, however, Theorem~\ref{anybktrivall} gives a strong negative result. While $\LK(\deltwo)$ must be closed downwards in the Turing degrees, $\KT(\deltwo)$ is cofinal in this structure. This separates $\KT(\deltwo)$ for $\LK(f)$ for any $f$, not just $\deltwo$ orders. 

\begin{cor}
For any function $f:\twow \rightarrow \N$ either $\LK(f)=2^{\omega}$ or $\KT(\deltwo) \nsubseteq \LK(f)$. 
\end{cor}

\begin{proof}
If there is an $B$ such that $B\notin \LK(f)$, then by Theorem~\ref{anybktrivall} there is an $A \geq_{T} B$ such that $A \in \KT(\deltwo)$. Now, since $\LK(f)$ is closed downwards under $\leq_{T}$ and $B\notin \LK(f)$, $A$ cannot be in $\LK(f)$. Thus, there is an $A$ that is in $\KT(\deltwo)$ but not in $\LK(f)$. 
\end{proof}

We know there are reals that are not in $\KT(\deltwo)$, so in particular $\KT(\deltwo)\neq 2^{\omega}$. This gives us the following corollary, which demonstrates that it is impossible to capture this notion of bounded initial segment complexity with any notion of bounded compressive power. 

\begin{cor}
There is no collection $F$ of functions $\twow \rightarrow \N$ such that $\LK(F)=\KT(\deltwo)$, where $\LK(F)=\bigcap\limits_{f\in F}\LK(f)$.
\end{cor}

In particular, since we know $\LK(\deltwo)\neq 2^{\omega}$, we have that $\KT(\deltwo)\nsubseteq \LK(\deltwo)$.

%In the other direction we get a strong negative result: there are $\deltwo$ (even recursive) orders $f$ such that for no $\deltwo$ order $g$ can $\KT(g)$ imply $\LK(f)$

\section{Other Closures}

Now, by Theorem~\ref{anybktrivall} we know that in general $\KT(g)$ and $\KT(\deltwo)$ are not closed downwards under $\leq_T$, so traditionally they would not be considered `lowness' notions. In the interests of defending our title, we examine some other reducibility notions under which these sets are closed downwards. First we show that for a stronger computational reducibility $\KT(\deltwo)$ \emph{is} closed downwards. $A$ is \emph{weak truth-table reducible} to $B$ (denoted $A \leq_{wtt} B$) if there is a Turing functional $\Phi$ and a recursive function $f$ such that $\Phi^B=A$ and for any $n$, the use of the computation of the $n$th bit of $A$ from $B$ (the largest bit of $B$ that is queried in the computation) is no more than $f(n)$. That is, not only can we use $B$ to compute $A$, but we have a recursive bound on how much of $B$ is needed to compute a given amount of $A$. It follows easily from the definition of $K$-triviality that $KT(0)$ is closed downwards under $\leq_{wtt}$ and we show that this closure is preserved under the weakening to $\KT(\deltwo)$. 

\begin{thm}\label{ktfoawttdown}
If $A \leq_{wtt} B$ and $B\in \KT(\deltwo)$, then $A\in \KT(\deltwo)$.
\end{thm}

\begin{proof}
Suppose $B$ can compute $A$ via Turing functional $\Phi$, with use $\phi^{B}$. For any recursive $f$, we can find a recursive function that majorizes $f$ and that is monotone increasing, so without loss of generality we can assume the we have an increasing recursive bound $f$ on the use of $\Phi^{B}$. Now, given $\restr{B}{f(n)}$ we can find $n$, since $f$ is recursive and injective, and then we can run $\Phi$ on this initial segment of $B$ to get the initial segment of $A$ of length $n$. Thus, to describe $\restr{A}{n}$ all we need is $\restr{B}{f(n)}$ and some constant that is a code for the functional $\Phi$, so we have $K(\restr{A}{n})\leq^{+}  K(\restr{B}{f(n)})$. We wish to show that for an arbitrary finite-to-one approximable function $g$, $K(\restr{A}{n})\leq^{+} K(n) +g(n)$. Given such a $g$, we define a new function $h$, by $h(n)= g(m)$, where $m$ is the greatest number such that $f(m) \leq n$. Finding this $m$ can be done recursively, so $h$ is also finite-to-one approximable. Thus, since $B$ is $K$-trivial up to $h$, $K(\restr{B}{f(n)})\leq^+ K(f(n)) +h(f(n))$. Now $f$ is recursive, so $K(n)=^{+} K(f(n))$, and by definition $h(f(n))=g(n)$, so finally we get $K(\restr{A}{n})\leq^{+} K(n)+g(n)$, as desired. 
\end{proof}

%extra
%We note that downwards closure under $\leq_{wtt}$ does not hold in general for $\KT(g)$ for arbitrary $g$. In the proof of Theorem~\ref{ktrivnotdown} we in fact build a weak truth-table reduction, since the use of $\Psi^\restr{A}{n_{3i}}=\restr{B}{n_i}$ can be given in advance. The construction is actually so simple that $B$ can also compute $A$ via a weak truth-table reduction  $\restr{\Theta_{s}^{\restr{B_{s}}{n_{i}}}}{n_{3i}} \downarrow =\restr{A_{s}}{n_{3i}}$, and so $\KT(g)$ is not a $wtt$-degree theoretic notion. For the interested reader, in fact it is easy to see these reductions can be made total, so this failure of closure holds even for truth-table reducibility. 

Another closure property we get for $\KT(\deltwo)$, in contrast to $\LK(\deltwo)$ as in \cite{herbert}, is that $\KT(\deltwo)$ is closed under effective join (the effective join of reals $A$ and $B$, denoted $A \oplus B$, is the real whose binary expansion is given by $A \oplus B (2n)=A(n)$ and $A \oplus B(2n+1)=B(n)$). The proof follows closely the proof that $\KT(0)$ is closed under effective join that was given by Downey, Hirschfeldt, Nies, and Stephan \cite{downhirschniessteph}. 

\begin{thm}\label{ktfoajoinclosed} 
For any reals $A$ and $B$, $A$, $B$ $\in\KT(\deltwo)$ if and only if $A \oplus B\in\KT(\deltwo)$. 
\end{thm}

\begin{proof}
Given an $A$ and $B$ in $\KT(\deltwo)$, we take an arbitrary finite-to-one approximable function $f$ and show that $K(\restr{A \oplus B}{n}) \leq^{+} K(n) +f(n)$. Without loss of generality, we can assume that $f(n)$ is monotonic. We note that it suffices to show this inequality holds for just the even $n$, since for any $\s$, $K(\s) = ^{+} K(\s \conc 1)=^{+} K(\s\conc 0)$ and for any $n$, $K(n+1)=^{+} K(n)$. 

We define a new finite-to-one approximable function $g(n)=\lfloor f(2n) /3 \rfloor$. Since $A$ and $B$ are in $\KT(\deltwo)$, in particular they are $K$-trivial up to this new $g$, so for some constants $b_A$ and $b_B$, for every $n$, we have $K(\restr{A}{n})\leq K(n) +g(n) +b_A$ and $K(\restr{B}{n})\leq K(n) +g(n) +b_B$. We let $b=\max\{b_A, b_B\}$. It is a theorem of Downey et al. (Theorem 5.5 in \cite{downhirschniessteph}) that there is a constant $\textbf{c}$ such that for any $k$ the cardinality of the set $S_{n,k}=\{\s: |\s|=n\ \& \ K(\s) \leq K(n) +k\}$ is no more than $2^{\textbf{c}}2^k$. Importantly, it does not depend on $n$. Thus, we know that $\restr{A}{n}$ and $\restr{B}{n}$ are both elements of the set $S_{n, g(n)+b}$, which is relatively small. Moreover, the set $S_{n, g(n)+b}$ is uniformly recursively enumerable in $K(n)$, $g(n)$, and $b$, so we can describe $\restr{A}{n}$ and $\restr{B}{n}$ relatively easily by giving their positions in the enumeration of this smallish set. 

Formally, we define a prefix-free machine $M$ that works as follows. On a string $\tau = 0^k\conc 1 \conc \s\conc \alpha \conc \beta$, where $\alpha$ and $\beta$ are both strings of length $k+\textbf{c}$, $M$ runs the universal machine $\U$ on $\s$ until it converges and then defines $n=\U(\s)$. $M$ interprets $\alpha$ and $\beta$ as binary representations of numbers less than $2^{\textbf{c}+k}$ (here it uses $\langle 0^{\textbf{c}+k}\rangle$ for the number 1 and $\langle 1^{\textbf{c}+k}\rangle $ for the number $2^{\textbf{c}+k}$), and waits for the $\alpha$th and $\beta$th strings of length $n$ to receive descriptions from the universal machine of lengths less than $|\s|+k$. When it finds these two strings, it outputs their effective join. 

%rho? for \s in tau=o^k\s\alpha\beta????????

$M$ is clearly partial recursive and its prefix-freeness follows from the prefix-freeness of $\U$ and the call to $\U$ on $\s$. If we take a $\tau=0^{g(n)+b}\conc 1 \conc \s \conc \alpha \conc \beta$ where $\s$ is a shortest description of $n$, and $\restr{A}{n}$ and $\restr{B}{n}$ are the $\alpha$th and $\beta$th strings of length $n$ to receive descriptions shorter than length $n+g(n)+b$, then$M(\tau)=\restr{A}{n} \oplus \restr{B}{n}= \restr{A\oplus B}{2n}$. This string $\tau$ has length $g(n)+b+1+K(n)+2(\textbf{c}+g(n)+b)$. Thus, by the universality of $\U$, we get that $K(\restr{A \oplus B}{2n}) \leq^{+} K(n)+ 3g(n) + 2\textbf{c}+ 3b +1$. Since $\textbf{c}$ and $b$ are constants that do not depend on $n$, and $K(n) =^{+} K(2n)$, we can rewrite this as $K(\restr{A \oplus B}{2n})\leq^{+} K(2n) +3g(n)\leq^{+} K(2n) +f(2n)$, by the definition of $g$. This suffices to show that $A\oplus B$ is in $\KT(f)$. 

For the other direction, for any $A$ and $B$, $K(\restr{A}{n})\leq^+ K(\restr{A \oplus B}{2n})$ and for all $n$, $K(n)=^{+}K(2n)$. Thus, if $A\oplus B \in \KT(\deltwo)$, then $K(\restr{A}{n})\leq^+ K(n)+f(2n)$ for any finite-to-one approximable $f$. To show that $A\in\KT(g)$ for a given finite-to-one approximable $g$, we simply take an $f$ such that $f(2n)=f(2n+1)=g(n)$ for all $n$. This is clearly finite-to-one approximable if $g$ is. By a symmetrical argument, $B$ must also be in $\KT(\deltwo)$. 
\end{proof}

By Theorems~\ref{ktfoawttdown} and ~\ref{ktfoajoinclosed}, we get that $\KT(\deltwo)$ is an uncountable ideal in the $wtt$-degrees, so it is not too far removed from being a legitimate `lowness notion.' One rather interesting side note is that $\KT(\deltwo)$ and $\LK(\deltwo)$ exhibit exactly the opposite behavior in terms of being Turing ideals. $\KT(\deltwo)$ is closed under $\oplus$ but for any real $A$ there is a real $B\in \KT(\deltwo)$ with $B\geq_T A$, while the set $\LK(\deltwo)$ is closed downwards under $\leq_T$ but for any $A$ there are $B$ and $C$ in $\LK(\deltwo)$ with $B\oplus C \geq_T A$ (so $\LK(\deltwo)$ is cofinal in the $T$-degrees under $\oplus$). This goes some way to suggest how each definition contributes to the various closure properties of $\KT(0)=\LK(0)$ and demonstrates how important Nies's Theorem is to our understanding of these properties. 

\section{Weak Reducibilities}
  
We now consider these sets of reals under other, weaker reducibility notions. A natural reducibility notion under which to consider the reals in $\KT(\deltwo)$ is $K$-reducibility, introduced by Downey, Hirschfeldt, and LaForte in \cite{kreduce}. We say $A$ is \emph{K-reducible} to $B$ ($A \leq_K B$) if for all $n$, $K(\restr{A}{n})\leq^{+} K(\restr{B}{n})$. Equivalently, we often say that $B$ is \emph{K-above} $A$ or that $A$ is \emph{K-below} $B$. From the definition it is immediate that $\KT(g)$ for any $g$ and $\KT(\deltwo)$ are closed downwards under $\leq_K$. Because a $K$-reduction does not need to have a concrete object as a witness (the way a Turing-reduction needs a Turing functional), it is not necessarily the case that the set of reals reducible to a given real will always be countable. In fact this is often not the case, as every Martin-L\"of random real has an uncountable lower cone in the $K$-degrees\cite{yudingdowney}. The following theorem shows that for reals in $\KT(\deltwo)$ at least, this is not the case. A precise characterization of those reals with countable lower $\leq_{K}$-cones is at this time still an open problem.

\begin{thm}[with F. Stephan]\label{lowerkcone}
If $A$ is in $\KT(\deltwo)$, then $A$ has a countable lower $\leq_K$-cone. 
\end{thm}

\begin{proof}
We show that if $A\in \KT(\deltwo)$ then $A$ is infinitely often $K$-trivial (i.e. for some constant $c$ there are infinitely many $n$ satisfying $K(\restr{A}{n})\leq K(n) +c$). Infinitely often $K$-trivial reals have been studied by Barmpalias and Vlek \cite{barmpaliasvlek}, and in particular they have shown that if $A$ is infinitely often $K$-trivial, then any $B\leq_{K} A$ is $\deltwo$ in $A$, and so $A$'s lower $\leq_K$-cone must be countable.

In fact, $A\in \KT(\log\log n)$ suffices to ensure that $A$ is infinitely often $K$-trivial. To show this, we assume $A\in\KT(\log \log n)$ and find infinitely many $n$ where $K(\restr{A}{n})\leq^+ K(n)$. 

Since $A\in\KT(\log \log n)$, for any $m$, $K(\restr{A}{2^m})\leq^+ K(2^m) +\log \log 2^m =^{+} K(m)+\log m$. Now, for any number $m$, $K(m)$ is always up to a constant less than $\log m +2 \log\log m$ (see, for example Chapter 2 of \cite{nies}), so we get that $K(\restr{A}{2^m})\leq^{+} 2\log m +2\log\log m$, and for large enough $m$ this quantity is less than $m$. 
Thus, there is some $\s\in\twow$ with $|\s|<m$ such that $\U(\s)=\restr{A}{2^m}$. Because $|\s|<m$, $\s$ can be interpreted as the binary representation of a natural number, $\text{num}(\s)$, with value less than $2^m$.
 Now for this $\s$, $K(\restr{A}{\text{num}(\s)})\leq^+ K(\s)$, since from a description for $\s$ one can run the universal machine on $\s$ to get $\restr{A}{2^m}$ and then compute $\text{num}(\s)$ and truncate this string to $\restr{A}{\text{num}(\s)}$. It is a recursive process to go from a binary string $\s$ to the natural number it is a binary representation of, so $K(\s)=^+ K(\text{num}(\s))$, and so $K(\restr{A}{\text{num}(\s)})\leq^+ K(\text{num}(\s))$. 
For every sufficiently large $m$, such a $\s$ exists, and they are necessarily distinct for distinct $m$, so $A$ is infinitely often $K$-trivial. Thus, by the result of Barmpalias and Vlek, the set $\{B : B\leq_K A\}$ is countable. 

\end{proof}

$K$-reducibility is a way to preorder reals by their relative initial segment complexities, there is an analogous reducibility notion for relative compressive power. We say a real $A$ is $LK$\emph{-reducible} to $B$ ($A \leq_{LK} B$) if up to an additive constant for all $\s\in\twow$, $K^{B}(\s)\leq^+ K^{A}(\s)$, that is, $A$ compresses strings at most as well as $B$ does. It is clear that the sets $\LK(f)$ and $\LK(\deltwo)$ are closed downwards under $\leq_{LK}$. Analogously to infinitely often $K$-trivial reals, we have reals that are \emph{weakly low for K}.

\begin{definition}
A real $A$ is \emph{weakly low for K} if there is a $c\in\N$ such that for infinitely many $\s\in\twow$, $K(\s)\leq K^{A}(\s)+c$. 
\end{definition}

Miller \cite{millerweaklowk} showed that these reals correspond to the low for $\Omega$ reals defined by Nies, Stephan, and Terwijn in \cite{lowforomega}. We call a real \emph{low for} $\Omega$ if for any $c$ there is an $n_c$ such that $K^A(\restr{\Omega}{n_c})<n_c-c$, where $\Omega$ is the measure of the domain of the universal machine $\U$. In the same paper Miller showed that these weakly low for $K$ reals had countable lower cones in the $LK$-degrees, and conjectured that these were in fact the only reals that do. Barmpalias and Lewis \cite{barmlewislkdegrees} settled this question in the affirmative. Following Theorem~\ref{lowerkcone} we could hope to show that the reals in $\LK(\deltwo)$ all have countable lower $LK$-cones, but unfortunately this is not the case. The rest of this section comprises a proof of the separation of $\LK(\deltwo)$ and weak lowness for $K$. One direction follows easily from the existence of Martin-L\"of random reals that are low for $\Omega$, since none of these can be in $\LK(\deltwo)$. The other direction is more complicated. 
\\

\begin{thm}\label{lkdeltwonotweaklowk}
There are reals in $\LK(\deltwo)$ that are not low for $\Omega$, and so not weakly low for $K$. 
\end{thm}

\begin{proof}
The construction will be similar to the one in the proof of Theorem~\ref{anybktrivall}. The idea is to build a perfect binary branching trees whose branching nodes will represent guesses as to which of the $\phies$ are finite-to-one approximations and alongside this to build a Kraft-Chaitin set to witness that each path is low for $K$ up to those $\phies$ which it guesses are finite-to-one approximations. To ensure that the path is not low for $\Omega$, we enumerate an oracle Kraft-Chaitin set using potential paths through the tree as oracles and giving short descriptions to initial segments of $\Omega$ relative to these paths. $\Omega$ is $\deltwo$, so we can approximate its initial segments recursively. Tension arises between trying to put short descriptions of initial segments of $\Omega$ onto paths through our tree while also trying to match descriptions relative to paths through the tree with almost-the-same-length descriptions by our oracle-free Kraft-Chaitin sets. The construction will generate a lot of waste mass into all the Kraft-Chaitin sets since we have at best an approximation to $\Omega$ and to the finite-to-one approximable $\phies$'s. This part will be more difficult than the proof above, since here we are trying to use an oracle-free Kraft-Chaitin set to pay for descriptions of strings matching those relative to an oracle. As we change the tree the oracles will change, so our opponent will get mass back with which to challenge us while the mass we spent will have been wasted. 

%%%%%%%%%%%%%%%%%%%%%%%%%%%%%%%%%%%%%%%%%%%%%%%%%%%%%%%%%%%%%%%%%%%55
%%%%%%%%%%%%%%%%%%%%%%%%%%%%%%%%%%%%%%%%%%%%%%%%%%%%%%%%%%%%%%%%%%%%%%%
Branchings in the tree $T$ will alternate between \emph{guessing nodes}, which are associated with guesses as to the behavior of some $\phie$, and \emph{compression intervals}, collections of $i$-many branchings (so $2^i$-many top-level nodes) on which, for some $n$, we place descriptions of possible $\restr{\Omega}{n}$ of length $n-i$. We distribute the descriptions as evenly as possible among the top-level nodes of a compression interval, to ensure that the measure of the domain of the machine we construct is finite with respect to each of the paths as an oracle.

%We will call the first kind of branching nodes \emph{guessing nodes} and the second kind \emph{coding nodes}, since they will eventually be used as coding locations for showing that we can join paths through $T$ above any given real. Note that unlike the constructions in Chapter 2 we will be attempting to keep paths through our subtree identical except for at coding locations, so we will also be able to talk about \emph{branching levels} and \emph{coding levels}.

In the proof of Theorem~\ref{anybktrivall} we waited till the various $\phies$ took values greater than $2i$; here we will use the sequence $c_i=10i^4$ for the same purpose. We use similar terminology to the previous proof. 

For a given $\alpha\in \twow$, we will say $e$ is a \emph{member} of $\alpha$ if $\alpha(e)=1$ and denote this $e\in \alpha$. For each $\alpha\in\twow$ we define a function that guesses that $\alpha$ is correct: $\psi_{\alpha,s}(\s)=\min\{\phies(\s): e\in \alpha\}$.

We will say a path $\rho$ though $T$ \emph{follows} a string $\alpha$ through the guessing nodes of $T$ if for each $i\leq|\alpha|$, $\eta_i \conc \alpha(i) \preceq \rho$, where $\eta_i$ is the $i$th guessing node that is an initial segment of $\rho$. Such a path is \emph{minimal} if it if minimal under the $\preceq$ relation (i.e. $|\rho|=|\eta_{|\alpha|}|+1$).

The requirements that we will try to meet are:
\begin{flalign*}
R_{\alpha}: &\text{ For all paths through } T \text{ that follow } \alpha \text{ at the \emph{guessing} nodes, there is a level}& \\
 &\text{where they all branch } (2|\alpha |)+1 \text{-many more times}&
\end{flalign*}
for all $\alpha \in \twow$, and
\begin{flalign*}
S^e_i: &\text{ For all } \s \text{ with } c_i \leq \phi_e(\s) < c_{i+1}, \text{ we have }K(\s) \leq^+ K^{A}(\s) + c_i \text{ for all } A \in[T]&
\end{flalign*}
for all $i,e \in \N$ with $i\geq e$. 
\begin{flalign*}
N_{\alpha}: &\text{ For any minimal path } \rho \text{ that follows } \alpha \text{ through the guessing nodes of $T$}& \\
& \text{ there is an extension, } \rho' \text{ to an } (|\alpha |+1) \text{st}&\\
& \text{ guessing node and an } m \text{ such that } K^{\rho'}(\restr{\Omega}{m})\leq^{+} m-|\alpha |&
\end{flalign*}
for all $\alpha \in \twow$ with $|\alpha|\geq 1$. 

Note that we only have $S^e_i$ requirements for $i \geq e$. This prevents $\phi_e$ from injuring the tree below the guessing node for $e$. 

We order the requirements $R_{\langle \rangle}, S^0_0, R_{\langle 0 \rangle}, R_{\langle 1 \rangle}, N_{\langle 0 \rangle}, N_{\langle 1 \rangle}, S^0_1, S^1_1, R_{\langle 00 \rangle}, R_{\langle 01 \rangle},$ $R_{\langle 10 \rangle}, R_{\langle 11 \rangle}, N_{\langle 00 \rangle}, N_{\langle 01 \rangle}, N_{\langle 10 \rangle},$ 
$N_{\langle 11 \rangle}, S^0_2, S^1_2, S^2_2, \hdots$. 

The various $S^e_i$ requirements will be concerned with different $\s$ throughout the construction as the approximations to the $\phies$ settle. We will say $S^e_i$ is $e$\emph{-responsible} for $\s$ if at stage $s$ we have $c_i \leq \phies(\s) < c_{i+1}$ and $\s$ is one of the length-lexicographically first $s$ elements of $\twow$. Note that for each $e$ for a given string $\s$ and stage $s$ at most one $S^e_i$ is $e$-responsible for $\s$ at $s$. As in the last proof, we will want to keep track of our guessing levels and to this end we will use a collection of markers $n_{\alpha,s}$. Each $n_{\alpha,s}$ will mark the the end of the compression interval where the paths that follow $\alpha$ through the first $|\alpha|$-many guessing nodes of $T_s$ have descriptions of $\restr{\Omega}{m}$ that are shorter by $|\alpha|$, and will be the next guessing level.  As before, to \emph{kill} a node is to make a commitment to never add nodes above it into $T$. Nodes in $T_s$ that have not been killed are \emph{living}. A node is a \emph{leaf node} at stage $s$ if neither of its successors is in $T_s$. 

We now give the strategies for satisfying each of our requirements. 

An $R_{\alpha}$ requirement \emph{requires attention} at a stage $s$ if the guessing level $n_{\alpha,s}$ is not defined. The strategy for meeting $R_{\alpha}$ is

\begin{enumerate}
%\item Let $\alpha^- \prec \alpha$ be such that $|\alpha^-|=|\alpha|-1$ (i.e., the immediate predecessor of $\alpha$). 
\item Let $n$ be some number larger than any seen before in the construction
\item For every living leaf node, $\eta$, of $T_s$ that follows $\alpha$ through the first $|\alpha|$-many guessing nodes, add the path $\eta\conc \beta \conc \gamma$ to $T_s$ to get $T_{s+1}$, where $|\eta|+|\beta|=n$, $\beta(i)=0$ for all $i$ where it is defined, for every $\gamma \in 2^{2|\alpha|+1}$. 
\item Let $n_{\alpha,s+1}=n+2|\alpha|$. This is now a guessing level. 

\end{enumerate}

An $N_{\alpha}$ requirement \emph{requires attention} at a stage $s$ if one of the minimal nodes $\rho$ that follows $\alpha$ through the guessing nodes of $T_s$ does not have an extension $\rho'\in T_s$ to an $|\alpha|+1$st guessing node for which the request $( \restr{\Omega_s}{m_{\alpha,s}}, m_{\alpha,s}-|\alpha|, \rho')$ has been put into $M$, for the current $m_{\alpha,s}$. The strategy for meeting this requirement is

\begin{enumerate}
%\item Find the smallest $n$ such that for all $t\leq s$, for all $\s$ of length $n$, $\psi_{\alpha,t}(\s)>c_|\alpha|$. 

\item If $m_{\alpha,s}$ is undefined, pick some $m>2|\alpha|+1$ that is also larger than anything seen so far in the construction and let $m_{\alpha,s}=m$. 

\item For the current $m_{\alpha,s}$, for the $|\alpha|+1$st guessing node, $\rho'$ in $T_s$ that follows $\alpha$ through the lower $|\alpha|$ guessing nodes and is the leftmost that has had the least amount of mass put into $M$ so far, put the request$(\restr{\Omega_s}{m_{\alpha,s}},m_{\alpha,s}-|\alpha|, \rho')$ into $M_{s+1}$. 
\end{enumerate}

%(so $c_i \leq \phies(\s) < c_{i+1}$ and $\s$ is one of the length-lexicographically first $s$-many strings)
An $S^e_i$ requirement \emph{requires attention} at a stage $s$ if there is a $\s$ that it is $e$-responsible for and there is a living partial path $\gamma$ in $T_s$ that, if it goes through at least $e$-many guessing nodes then it takes the `$1$' branch after the $e$th one, and we have $K_{s}^{\gamma}(\s)+\phies(\s)$ is less than the shortest description of $\s$ in $L_{e,s}$. This means that the shorter description of $\s$ is on a path that either has not reached a guessing node for $\phies$ or is guessing that it is a finite-to-one approximation, so we will need to act. The strategy for meeting this requirement is

\begin{enumerate}
\item Find the length-lexicographically least $\s$ and for this $\s$ the length-lexicographically least $\gamma$ that are causing $S^e_i$ to require attention. By the choice of these as length-lexicographically least, we must have that the use of the computation $\U_s^{\gamma}(\tau)\downarrow=\s$ that is causing $S^e_i$ to act is $|\gamma|$. 

\item For this $\gamma$, let $\alpha$ be maximal such that $\gamma$ follows $\alpha$ through the guessing nodes of $T_s$. In other words, $\alpha$ is the collection of guesses that are being made on the path $\gamma$. 

\item If $|\alpha| <e$, then $\gamma$ has not guessed about the behavior of $\phies$, but we know $i \geq e$, so we can afford to pay for a description of $\s$ on this part of the tree anyway. Put a request $( \s, K_{s}^{\gamma}(\s)+c_i )$ into $L_{e,s}$ to get $L_{e,s+1}$. 

\item Otherwise, $|\alpha| \geq e$, and since $S^e_i$ requires attention, we must have $\alpha(e)=1$. We now consider whether $|\alpha|\leq i+1$, in order to check whether $\gamma$ has too many guessing nodes and will cause us to injure the tree. If $|\alpha|\leq i+1$, then we have not yet branched for the $i(i+1)$th time (where the $(i+1)$st guessing node would be) so we can pay. Put a request $( \s, K_{s}^{\gamma}(\s)+c_i )$ into $L_{e,s}$ to get $L_{e,s+1}$. 

\item Otherwise, $|\alpha|>i+1$, so $\gamma$ is longer than the $(i+1)$st guessing level. Since $\phies(\s)<c_{i+1}$, this is too high, so we

\begin{enumerate}

\item Injure $R_{\restr{\alpha}{i+1}}$ and run the Injury Subroutine for it.

\item Let $T_{s+1}=T_s$, $L_{e,s+1}=L_{e,s}$.

\end{enumerate}

\end{enumerate}

The Injury Subroutine for an $R_{\alpha}$ strategy at stage $s$ is

\begin{enumerate}

%\item Find the living node $\eta$ at height $n_{\alpha,s}$  and the string $\rho$ such that $\eta\conc \rho$ is a living leaf node of $T_s$, $\eta$ follows $\alpha$ through the guessing nodes of $T_s$, and $\sum\limits_{\tau:\U_s^{\eta \conc \rho}(\tau)\downarrow, \U_s^{\eta}(\tau)\uparrow} 2^{-|\tau|}$ is maximal. If there is more than one pair $(\eta,\rho)$, choose the leftmost.

%\item For every node $\zeta$ at height $n_{\alpha,s}$ that follows $\alpha$ through the first $|\alpha|$-many guessing nodes, keep $\zeta\conc \rho$ alive. Kill every other extension of $\zeta$. Set all $R_{\beta}$ for $\beta\succeq \alpha$ to requiring attention (i.e. set $n_{\beta,s+1}$ to be undefined). 
\item For every minimal path $\eta$ that follows $\alpha$ through the guessing nodes of $T_s$, find the living leaf node $\eta' \succ \eta$ such that $\sum\limits_{\tau:\U_s^{\eta'}(\tau)\downarrow, \U_s^{\eta}(\tau)\uparrow} 2^{-|\tau|}$ is maximal. If there is more than one, take $\eta'$ to be the leftmost. 

\item For every pair $(\eta,\eta')$ found above, keep $\eta'$ alive in $T_s$ and kill all other extensions of $\eta$. 

\item For every $\beta \succeq \alpha$, set $R_{\beta}$ to requiring attention (i.e. set $n_{\beta,s+1}$ to be undefined) and set $m_{\beta,s+1}=m_{\beta,s}+1$. This injures all these $R_{\beta}$ and $N_{\beta}$

\end{enumerate}

The skeleton of the Construction is

\textbf{Stage 0}: Set $T_0=\emptyset$, $M_0= \emptyset$, $L_{e,0}=\emptyset$ and for every $e$,  and $n_{\alpha,0}$ undefined for all $\alpha$.\\

\textbf{Stage $s+1$}:
\begin{enumerate}
\item Compute $\phi_{e,s+1}(\s)$ and $K_{s+1}^{\eta}(\s)$ for all living branches $\eta$ in $T_s$, the first $s+1$-many $\s$'s, and $e\leq s+1$.

\item In order of priority, run the strategy for each of the first $s+1$-many requirements that require attention, including executing the Injury Subroutine as necessary. 

\item For any $n_{\alpha,s}$ or $m_{\alpha,s}$ that were not affected, set $n_{\alpha,s+1}=n_{\alpha,s}$ and $m_{\alpha,s+1}=m_{\alpha,s}$
\end{enumerate}

Now let $T=\bigcup\limits_s T_s$, $L_{e}=\bigcup\limits_s L_{e,s}$, $n_{\alpha}=\lim\limits_s n_{\alpha,s}$, $M=\bigcup\limits_s M_s$, $m_{\alpha}=\lim\limits_s m_{\alpha,s}$. This completes the construction. The verification follows. 

As in the proofs of Theorem~\ref{anybktrivall}, we will not be able to ensure that all requirements are satisfied, but only those that are correct about their guesses. We would like to show that every path through the subtree of $T$ generated by all the correct guesses about the $\phies$ is in $\LK(\deltwo)$ and at least one of them is not low for $\Omega$. We call $\alpha$ \emph{correct} if for every $e\leq |\alpha|$, $e\in\alpha$ if and only if $\phies$ is a finite-to-one approximation.

\begin{lemma}\label{finiteinjuryall}
For all correct $\alpha$, $R_{\alpha}$ and $N_{\alpha}$ are injured only finitely often. 
\end{lemma}
\begin{proof}
By construction, for any $\alpha$, the requirements $R_{\alpha}$ and $N_{\alpha}$ can only be injured by $S^e_i$ requirements with $e\in\alpha$ and $i<|\alpha|$, of which there are only finitely many. Assuming $\alpha$ is correct, the only $\phi_e$ that can cause these injuries are then indeed total finite-to-one approximations. Since $R_{\alpha}$ and $N_{\alpha}$ are always injured together, it will suffice to show that $R_{\alpha}$ is injured only finitely often. 

To derive a contradiction, first let us assume there is some $R_{\alpha}$ for a correct $\alpha$ that is injured infinitely often, and take $\alpha$ to be a minimal such string. Each of the $S^e_i$'s that can injure $R_{\alpha}$ only ever has $e$-responsibility for finitely many $\s$ since $\phi_e$ is a finite-to-one approximation, so there must be at least one $\s$ that is the cause of infinitely many injuries to $R_{\alpha}$. Let us take the length-lexicographically least such $\s$. 

Let us assume we are at a stage $s$ such that $\phies(\s)$ and $K_s(\s)$ have settled and such that no $R_{\beta}$ for $\beta \prec \alpha$ will ever be injured again. Since $\s$ causes infinitely many more injuries to $R_{\alpha}$, it must be the case that $S^e_{i}$ has $e$-responsibility for $\s$ for all stages $t \geq s$, for $i=|\alpha|-1$. 

Now, each run of the Injury Subroutine for $R_{\alpha}$ at some stage $t$ will, for each minimal $\eta$ that follows $\alpha$ through the guessing nodes of $T_t$, keep at least the most massive branch above $\eta$ alive and kill all other branches above $\eta$. There are always $1\cdot 4\cdot 16 \hdots 2^{2|\alpha|}$ many living nodes at height $n_{\alpha,t}$ that follow $\alpha$ through the guessing nodes, and a run of the Injury Subroutine for $R_{\alpha}$ extends each of these to a leaf node in the way that maximizes the mass placed along it. Since, by assumption, no earlier $R_{\beta}$ will ever be injured again, the tree below these paths never change so this mass is never lost. We are at a stage $t$ such that $K_t(\s)$ has already converged, so each injury caused by $\s$ must be caused by our finding a description of $\s$ of length less than $K(\s)$ along one of these paths. That is, at least $2^{-K(\s)}$ much mass must converge on one of the paths at height $n_{\alpha,t}$ that follow $\alpha$ through the guessing nodes. This is a fixed amount of mass that is added infinitely often to a finite number of oracles, so the measure of the domain of $\U$ relative to one of these oracles is infinite. This is a contradiction.

Thus, $R_{\alpha}$ (and so $N_{\alpha}$) can only be injured finitely often.

%For the branches that are kept alive, let $m$ be the minimal amount of mass that converged above any $\eta$ (there are finitely many $\eta$ that follow $\alpha$ through the guessing nodes, so the minimum exists)
\end{proof}

\begin{lemma}\label{reqsatisfiedall}
For all correct $\alpha$ and all $e$ such that $\phies$ is a total finite-to-one approximation, the requirements $R_{\alpha}$, $N_{\alpha}$, and $S^e_i$ are all eventually satisfied. 
\end{lemma}

\begin{proof}
Some requirements may cause infinitely many injuries to requirements above them in the ordering or require attention infinitely often. However, at any stage $s$ of the construction we allow any of the first $s$ requirements to act, and our actions affect different parts of the tree $T_s$ (the paths that follow different $\alpha$'s through the guessing nodes), so the poorly behaved requirements will not interfere with our actions in satisfying the correct ones. 

By the above lemma, for correct $\alpha$, the requirements $R_{\alpha}$ and $N_{\alpha}$ are only injured finitely often. After the last injury, $R_{\alpha}$ will need to act once more before it is satisfied, while $N_{\alpha}$ may need to act several times as it waits for $\restr{\Omega}{p_{\alpha}}$ to converge ($p_{\alpha}$ only changes when $N_{\alpha}$ gets injured). Eventually this happens and after that stage we will put a description of $\restr{\Omega}{p_{\alpha}}$ onto one of the relevant paths and satisfy $N_{\alpha}$ permanently.

If $\phies$ is a finite-to-one approximation, then, also by the proof above, the requirement $S^e_i$ can only cause finitely many injuries. It's actions that do not cause injuries are just those in steps 3. and 4. of its strategy and these consist of putting a request $(\s, K^{\gamma}(\s)+c_i)$ into $L_e$, so some $\s$ it is $e$-responsible for and some partial path $\gamma$ through $T_s$. Since $\phies$ is a finite-to-one approximation, there are only finitely many $\s$ for which $S^e_i$ is ever $e$-responsible. For each of these, there are only finitely many requests to put into $L_e$, since we only need to put new ones in if the new $K^{\gamma}(\s)$ is less than all previous ones. Thus, $S^e_i$ will need to act only finitely often.  

%It then follows that after a certain stage $t$, whenever $S^e_i$ requires attention the action we take to satisfy it is to put a request for a description of some $\s$ into $L_e$. There are only finitely many $\s$ that $S^e_i$ ever has $e$-responsibility for, and each time we put a request for a description of one of these $\s$ into $L_e$, the minimal length of a new description of $\s$ that would cause $S^e_i$ to require attention decreases by at least $1$. Every time $S^e_i$ acts it must be in response to some amount (at least $2^{-K(\s)}$, after this has converged) of mass being added to one of the branches in $T_s$.  Even if requirements below $S^e_i$ cause infinitely many injuries, at least this much mass must be preserved on branches that are saved by runs of the Injury Subroutine. Since there are only finitely many paths that $S^e_i$ is concerned with, eventually either $S^e_i$ sees the shortest description of $\s$ converge on one of these and puts a request into $L_e$, or injuries caused by earlier requirements fix so much mass on these branches that no more short descriptions of $\s$ \emph{can} appear on them while keeping the measure of the domain of $\U$ relative to these branches below $1$. In either case, $\s$ will no longer cause $S^e_i$ to require attention, and since this happens for each of the finitely many $\s$ that $S^e_i$ has $e$-responsibility for, eventually $S^e_i$ will be satisfied permanently.  

\end{proof}

Lemmas~\ref{finiteinjuryall} and ~\ref{reqsatisfiedall} give us that the strategies relevant to the construction of the true subtree will eventually stop acting, but it remains to be shown that the Kraft-Chaitin sets enumerated by these strategies have bounded mass (and so produce the required machines). We start with the machine $\mathcal{M}$ whose existence will be witnessed by $M$.

\begin{lemma}
For every real $A$, the sum $\sum\limits_{(\s,p,\eta)\in M, \eta \prec A}2^{-p}<2$. 
\end{lemma}

\begin{proof}
To prove this lemma we consider the amount of mass a given $N_{\alpha}$ requirement can contribute to $M$ for an oracle $A$. It is clear that for different $\alpha$ of the same length, the sets of oracles the $N_{\alpha}$'s use will be disjoint, since they will be paths through $T$ that take different directions at at least one of the guessing nodes. Thus, it suffices to show that a given $N_{\alpha}$ will add at most $2^{-|\alpha|+1}$ to any oracle, and so a path that receives mass from many $N_{\alpha}$'s will receive no more than $2$ total mass. 

Let us first fix an $\alpha$ of length at least $1$ and let $d=|\alpha|$. For any stage $s$ in the construction where $N_{\alpha}$ is active, we have some $m_{\alpha,s}$ which is always larger than $2|\alpha|+1$. Now, there are na\"{i}vely $2^{m_{\alpha,s}}$ many possible strings that could be $\restr{\Omega}{m_{\alpha,s}}$. When $N_{\alpha}$ requires attention, it puts a request $(\restr{\Omega}{m_{\alpha,s}}, m_{\alpha,s}-d,\eta)$ into $M_s$ for a node $\eta$ in $T_s$ at level $n_{\alpha,s}$ that follows $\alpha$ through the guessing nodes which has had the minimum amount of mass already placed on it by $N_{\alpha}$. This contributes $2^{-m_{\alpha,s}+d}$ much mass to that path. Without any injuries, this strategy would distribute $2^{m_{\alpha,s}} \cdot 2^{-m_{\alpha,s}+d}=2^{d}$ much mass evenly among the $2^{2d}$ many branches in this compression interval, giving $2^{-d}$ to each. Of course, injuries will complicate matters. A run of the Injury Subroutine will fix an amount of this mass onto each of the paths that follow $\alpha$ through the guessing nodes of $T_{s+1}$, which will be initial segments of the new nodes at $n_{\alpha,s+1}$ (after $R_{\alpha}$ acts again). In principle, this is fine. We now have $2^{2d}$ many nodes at the new $n_{\alpha,s+1}$ and $N_{\alpha}$ has put some amount $p$ of mass onto an initial segment of all of them and will continue by sharing its remaining mass equally over them. In the ideal case this situation is no different from continuing to share the mass over the nodes before the injury. They each had (almost) $p$ much mass already placed on them by $N_{\alpha}$, and the fact that $N_{\alpha}$ is continuing to act means whatever potential initial segments of $\Omega$ they paid for have since been rejected. 

The trouble, of course, is in the `almost.' If some node at $n_{\alpha,s}$ had received more mass than the others before the injury occurred, it is possible for the Injury Subroutine to have picked a path through that node to keep, in which case slightly more mass than has now converged on the initial segment of the new nodes at $n_{\alpha,s+1}$. In the worst case, $N_{\alpha}$ may act only once between injuries, each of which keeps alive the only path that $N_{\alpha}$ has added mass to. This will concentrate all the mass that $N_{\alpha}$ has to distribute onto a single path. For this reason we increment $m_{\alpha,s}$ by $1$ every time $N_{\alpha}$ is injured. We note that this does not affect the amount of mass that $N_{\alpha}$ has left to distribute; once we see that $\gamma\neq\restr{\Omega}{p}$ we know that neither extension of $\gamma$ is $\restr{\Omega}{p+1}$. The difference in mass that $N_{\alpha}$ can have added to paths before an injury is always at most $2^{-m_{\alpha,s}+d}$, since it always seeks to distribute the mass evenly, and it does so in quanta of $2^{-m_{\alpha,s}+d}$. Thus, even if injuries are selecting the paths with slightly higher than average mass, no more than $\sum_s 2^{-m_{\alpha,s}+d}$ can accumulate on a path above the average of $2^{-d}$. This sum is bounded by $2^d\cdot 2^{-m_{\alpha,s_0}+1}$, where $s_0$ is the stage at which $N_{\alpha}$ is initialized. Since we always choose our first $m_{\alpha,s_0}$ to be larger than $2d+1$, this term is less than $2^{-d}$. Thus, the total amount that $N_{\alpha}$ can have contributed to a path is $2^{-d+1}$, and so the total mass in $M$ for any oracle is bounded by $\sum_d 2^{-d+1}=2$.

\end{proof}

Since $M$ is a legitimate oracle Kraft-Chaitin set and the $N_{\alpha}$ requirements are satisfied for all correct $\alpha$, there will be an infinite path $A$ through $T$ that is not low for $\Omega$. $A$ will be the path that always guesses correctly as to the behavior of $\phies$ at the $e$th guessing node, and between guessing nodes follows the path through the compression intervalfor which $M$ has a short description of an actual initial segment of $\Omega$. 

All that remains to be shown is that this path is actually in $\LK(\deltwo)$, that is, that the sets $L_e$ that we construct actually witness the existence of machines ensuring that $K(\s)\leq^{+} K^{A}(\s)+\phi_e(\s)$ for all $\s$, for all paths through $T$ that guess that $\phi_e$ is finite-to-one approximable. This is the most complicated part of the proof, since the mass paid into $L_e$ can be wasted by injuries to the construction.

\begin{lemma}
For every $e$, $\sum\limits_{(\s,n)\in L_e}2^{-n}<1$.
\end{lemma}

\begin{proof}
We consider separately mass that is put into $L_e$ by the actions of each $S^e_i$ requirement. First, we fix an $e$ and $i$ in $\N$ with $i\geq e$. We note that we do not require $\phies$ to be a finite-to-one approximation; in the event that it is not, actions of $S^e_i$ strategies may cause infinitely many injuries to the part of the tree that guesses that it is, and then $L_e$ will witness that the finitely many infinite branches on this part of the tree are all low for $K$.

% For convenience, then, we define 
%\begin{multline*}T^e_{s} = \{ \s \in T_s: \text{if } \alpha \text{ is maximal such that } \s \text{ follows } \alpha \text{ through the guessing nodes of } T_s,\\ \text{ then either } |\alpha |<e \text{ or } \alpha(e)=1 \}, \end{multline*}
%and we let $T^e=\bigcup\limits_{s} T^e_{s}$. 

Now, $S^e_i$ requires attention if for some $\s$ for which it has $e$-responsibility, it sees a new shorter description of $\s$ converge on some path $\gamma$ through $T^e_{s}$ (let us take $\gamma$ minimal to cause this). However, $S^e_i$ only puts a request into $L_e$ if $\gamma$ goes through no more than $i+1$ guessing nodes, and otherwise it causes an injury. Therefore, when examining the mass contributed by $S^e_i$ it suffices to consider the finite initial segment of $T_{s}$ given by 

$$T_{i,s}=\{ \s \in T_s: \s \text{ goes through no more than } i+1 \text{-many guessing nodes}\}$$

We only put requests into $L_e$ is response to computations converging on parts of the tree that are either below the $e$th guessing level, or follow the `$1$' path at the $e$th guessing node itself, but for getting a rough upper bound on the mass of $L_e$, we can ignore this and consider the full subtree $T_{i,s}$.

The subtree $\Tis$ contains the guessing nodes at $n_{\alpha,s}$ for every $\alpha\in\twow$ with $|\alpha|\leq i+1$, and all the compression intervals between these nodes. For a given $\alpha$ with $|\alpha|\leq i+1$, for each minimal $\eta$ that follows $\alpha$ through $\Tis$, there are $2^{2|\alpha|}$ branches in the compression interval above $\eta$ that reach level $n_{\alpha,s}$. 

We will consider these branches as `reservoirs' of mass, and descriptions converging using one of these branches as an oracle as mass getting added to the corresponding reservoir. If the use of a computation is exactly one of these minimal $\eta$'s, we can consider that much mass being added to \emph{each} of $\eta$'s reservoirs. An injury will cause many of these branches to be killed, so the mass will be spilled out, but the most massive branch will be kept alive and, after enough $R$ requirements act again, the mass from that branch will end up in a lower reservoir for a shorter $\alpha'$.  It is important to note that for any $\alpha$ the reservoirs corresponding to $\alpha$ are all end extensions of reservoirs for $\alpha^{-}$, the immediate predecessor of $\alpha$, and so the total mass in any collection of nested reservoirs must be no more than $1$.

The number of reservoirs associated with an $\alpha$ is fixed throughout the construction, although we may have to wait for $R$ requirements to act to replace reservoirs that were emptied. For $\alpha$ with $|\alpha|=1$, there are $4$ reservoirs, since the compression interval has length $2$. Each of these has reservoirs above it for some $\alpha$ of length $2$, and the compression interval for these will have length $4$, so there will be $4\cdot2^{4}=64$ of these. In general, for $|\alpha|=i$, there will be 

$$\prod\limits_{j=1}^{i} 2^{2j}=2^{\sum\limits_{j=1}^{i} 2j}=2^{i^2+i}$$

reservoirs at level $i$.

When considering the contributions of $\Sei$ to $L_e$ we can consider only the reservoirs at level $i$. For any descriptions that converge lower in the tree $\Tis$ for $\s$ that $\Sei$ has $e$-responsibility for, we can instead put the corresponding amount of mass into each of the $i$-level reservoirs above the actual use of the computation, since this has the same effect on the subsequent amount of mass that can be put into these reservoirs. Then to attain a rough bound on the amount of mass that $\Sei$ puts into $L_e$, we can make the simplifying assumption that $S^e_i$ will pay for all the mass that passes through the reservoirs at level $i$, at a rate of $2^{-c_i}$ (the largest this can be without causing an injury). Now all that remains is to find a bound for the amount of mass that can pass through these reservoirs. 

As we said above, any injury to a relevant $R_{\alpha}$ will spill the mass from all but the most massive reservoir, and pour this saved mass into a reservoir below. Thus, in the worst, impossible, case, $\Sei$ could have to pay for all the $i$-level reservoirs being filled with $1$ total mass each, then an injury could empty all but one of these, and pass that $1$ down to the $i-1$ level, and this could repeat till all the $i-1$-level reservoirs are full. Then an injury could pour out all but one of these and fill one of the $i-2$-level reservoirs, and this larger process could repeat till all the $i-2$-level reservoirs are filled. Continuing like this, $\Sei$ could be forced to pay for the mass that is used to fill \emph{all} the reservoirs at all levels up to and including the $i$th one, while spilling as much as possible at every step. For each $1$st level reservoir, we would have to fill every $2$nd level one, and for each of these we would have to fill every $3$rd level one, continuing until level $i$. This gives us that the mass required, just to fill all the $1$st level reservoirs is

$$\prod_{j=1}^{i}2^{j^2+j},$$

and thus to fill all the reservoirs up to level $i$ would require

$$\sum_{k=1}^{i}\prod_{j=k}^{i}2^{j^2+j}$$

much mass to pass through the $i$-level reservoirs. 

We approximate a very rough upper bound for this:

\begin{align*}
\sum_{k=1}^{i}\prod_{j=k}^{i}2^{j^2+j} &\leq \sum_{k=1}^{i}\prod_{j=1}^{i}2^{j^2+j} \leq i\cdot 2^{\sum_{j=1}^{i}j^2+j}\\
&\leq 2^i \cdot 2^{[(i^2+i)(2i+1)/6]+[(i^2+i)/2]} \leq 2^{i+(i^2+i)(2i+2)}\\
&\leq 2^{2i^3+4i^2+3i}
\end{align*}

Thus, the amount of mass the $\Sei$ puts into $L_e$ is bounded by $2^{2i^3+4i^2+3i}\cdot 2^{-c_i}$. Since we have taken $c_i=10i^4$, this is the same as $2^{2i^3+4i^2+3i-10i^4}$, which is always no more than $2^{-i}$. Since this is the contribution for each $S_{e,i}$, the total mass of requests that go into $L_e$ for all $i$ is less than $\sum_i 2^{-i} =1$. 
\end{proof}
%of lemma

This completes the proof of Theorem~\ref{lkdeltwonotweaklowk}

\end{proof}
%of theorem

\section{Further Questions}

Theorem~\ref{lkdeltwonotweaklowk} and the result of Barmpalias and Lewis \cite{barmlewislkdegrees} show that not every $\deltwo$-bounded low for $K$ real has a countable lower $\leq_{LK}$-cone, but the question of which of these reals do have a countable lower $\leq_{LK}$-cone remains open. At present, the only reals we know are in the intersection are the $K$-trivials themselves. 
\begin{?}
Can we characterize the reals that are both $\deltwo$-bounded low for $K$ and weakly low for $K$(i.e., that have countable lower $\leq_{LK}$-cones? Are there any that are not $K$-trivial?
\end{?}

The analysis carried out in this paper was entirely in terms of prefix-free Kolmogorov complexity, but there are analogous notions in terms of plain Kolmogorov complexity (where the domains of decoding machines are not required to be prefix-free) that can also be weakened to $\deltwo$-bounded versions. In the case of $C$-triviality and lowness for $C$, we know by results of Chaitin \cite{chaitin} that these notions coincide with each other and contain only the recursive sets. So far we know nothing about the $\deltwo$-bounded versions. 
\begin{?}
What can we say about $\deltwo$-bounded lowness for $C$ or $C$ triviality?
\end{?}

Finally, we have considered here reals with initial segment complexity or compressive power bounded by \emph{all} $\deltwo$ orders. It  may be interesting to consider the internal structure of the various bounded notions, i.e., $\LK(f)$ and $\KT(g)$ for various $f$ and $g$. Many of the results for the $\deltwo$-bounded notions carry over trivially, for instance, the cofinality in the Turing degrees of $\KT(g)$ for any $\deltwo$ order $g$, but the theorems in Section 3 do not necessarily carry over, as the proofs depended on applying bounded initial segment complexity for different bounds. Clearly for some choices of $f$ we have $\LK(f)=2^{\omega}$ and for others it is much smaller (and similarly with $\KT(g)$), but it is open whether there is a strict cutoff between the two cases.

\begin{?}
What can we say about the sets $\LK(f)$ and $\KT(g)$ for single $\deltwo$ orders $f$ and $g$? What is the structure of these sets under $\leq_K$ or $\leq_LK$?
\end{?}

%Reals for which this function has a finite $\liminf$ are called \emph{weakly low for} $K$, and Miller has shown in \cite{millerweaklowk} that these coincide exactly with the low for $\Omega$ reals defined by Nies, Stephan, and Terwijn \cite{lowforomega}. 

\bibliographystyle{acm}
\bibliography{weaklowcomp}{}

\end{document}